\documentclass[10pt]{amsart}
\usepackage[utf8]{inputenc}
\usepackage[T1]{fontenc}
\usepackage[frenchb, english]{babel}
\title[On the structural properties of an efficient feedback law]{On the structural properties \\of an efficient feedback law}
\author{Ambroise Vest}
\address{Institut de Recherche Math\'ematique Avanc\'ee,  Universit\'e de Strasbourg\\7 rue Ren\'e Descartes, 67084 Strasbourg C\'edex, France}
\email{ambroise.vest@math.unistra.fr}
\date{February 25, 2013}
\subjclass[2000]{Primary 93D15; Secondary 93C05, 47D06, 93C20}
\keywords{stabilizability by feedback, linear distributed system, partial differential equation}
\usepackage{amsmath, amssymb, amsthm}
\usepackage{graphicx}
\usepackage{tikz}
\usetikzlibrary{arrows,shapes,trees,calc}
\usepackage{empheq}

\newcommand{\R}{\mathbb{R}}

\newcommand{\Om}{\Omega}
\newcommand{\om}{\omega}
\newcommand{\Lom}{\Lambda_{\omega}}
\newcommand{\ILom}{\Lambda_{\omega}^{-1}}
\newcommand{\ILomT}{\Lambda_{\omega,T}^{-1}}

\newcommand{\LomT}{\Lambda_{\omega,T}}
\newcommand{\LomTO}{\Lambda_{\omega,T_0}}
\newcommand{\RomT}{R_{\omega,T}}
\newcommand{\Tom}{T_{\omega}}
\newcommand{\eom}{e_{\omega}}
\newcommand{\eomT}{e_{\omega,T}}
\newcommand{\la}{\langle}

\newcommand{\ra}{\rangle}
\newcommand{\wJ}{\widetilde{J}}
\newcommand{\wA}{\widetilde{A}}
\newcommand{\wu}{\widetilde{u}}
\newcommand{\Do}{\mathcal{D}}
\newcommand{\Cont}{\mathcal{L}}
\newcommand{\ud}{\,\mathrm{d}}
\newtheorem{thm}{Theorem}[section]

\newtheorem{prop}[thm]{Proposition}

\theoremstyle{definition}
\newtheorem*{deff}{Definition}

\theoremstyle{remark}
\newtheorem*{rk}{Remark}

\theoremstyle{remark}
\newtheorem*{rks}{Remarks}

\theoremstyle{remark}

\theoremstyle{remark}

\theoremstyle{remark}
\newtheorem*{ex}{Example}
\begin{document}
\begin{abstract}
We investigate some structural properties of an efficient feedback law that stabilize linear time-reversible systems with an arbitrarily large decay rate. After giving a short proof of the generation of a group by the closed-loop operator, we focus on the domain of the infinitesimal generator in order to illustrate the difference bewteen a distributed control and a boundary control, the latter being technically more complex. We also give a new proof of the exponential decay of the solutions and we provide an explanation of the higher decay rate observed in some experiments.
\end{abstract}
\maketitle
\section{Introduction}
Given an abstract evolutionary problem
\begin{equation}\label{openloop}
\begin{cases}
x'(t)=Ax(t)+Bu(t),\quad t\ge 0,\\
x(0)=x_0,
\end{cases}
\end{equation}
where $A$ is the generator of a group, $B$ is a control operator and $u(t)$ is a control function, many works were devoted to the obtention of feedback operators
\begin{equation}\label{feedback}
u(t)=Fx(t)
\end{equation}
making the system \eqref{openloop}-\eqref{feedback} exponentially stable.

While the finite-dimensional case and more generally the case of a bounded operator $B$ were conveniently treated by semigroup method (see e.g., \cite{Lukes1968},\cite{Kleinman1970},\cite{Slemrod},\cite{Russell1979}), the case of unbounded control operators presented serious technical obstacles.

However, since unbounded control operators appear in the important problems of boundary or pointwise control, this question was studied extensively. Following several explicit constructions of stabilizing boundary feedbacks (see e.g., \cite{QR1977},\cite{Chen1979},\cite{Lagnese1983},\cite{LT1987}), J.-L. Lions \cite{Lions1988} gave a general approach that relies on the exact controllability of the system \eqref{openloop} and the theory of infinite-dimensional Riccati equations. In these works, the decay rate of the solutions cannot be arbitrarily large.

In \cite{K97}, V. Komornik  simplified the approach of \cite{Lions1988} by constructing feedbacks yielding prescribed decay rates. In a preceding work \cite{AV1}, we gave a rigourous justification of this method.

The purpose of this paper is twofold. In section \ref{WP}, we give a shorter and more direct justification of the well-posedness of the closed-loop system, avoiding the use of variation of constants formulas. Then, we exhibit a crucial difference between the cases of bounded and unbounded control operators. While in the bounded case the infinitesimal generator is just $A+BF$, in the unbounded case, it is an extension of $A+BF$. We show this by an explicit example. This essential difference may also explain why boundary control problems are usually technically more difficult to treat than internal control problems.

Our results also suggest  to have another look at the decay  rate of the solutions. In section \ref{Decay}, after giving a shorter proof  of the exponential decay of the solutions, we provide an explanation of an interesting phenomenon observed in \cite{Briffaut1999} and \cite{BJCR2004}:  in numerical simulations and physical experiments perfomed with the feedback introduced in \cite{K97}, the decay rate is usually twice as big as the one obtained theoretically. 

\begin{rks}${}$
\begin{itemize}
\item
The problems that we have in mind are time-reversible and generally model oscillating systems like waves or plates (see the examples given in \cite{K97}). This feedback has also been used to stabilize Maxwell equations in \cite{K98}, elastodynamic systems in \cite{AK1999} and partially observable systems in \cite{KL05}. 
\item
Another explicit feedback stabilizing a vibrating beam with arbitrarily large decay rates has been given recently in \cite{SGK2009}.
\end{itemize}
\end{rks}
\section{Hypotheses and notations}\label{Hyp}
Given a normed vector space $X$, $\|.\|_X$ denotes its norm, $\la.\,,.\ra_{X',X}$ denotes the duality pairing between $X$ and its dual $X'$. 
\footnote{
Sometimes we will omit the name of the spaces below the brackets in order to enlight the notations. 
}
The quantity $(.\,,.)_X$ represents a scalar product and $\|.\|$ represents the norm of a continuous linear operator beween two normed vector spaces, depending on the context.

The state space $H$ and the control space $U$ are Hilbert spaces. We denote by $H'$ and $U'$ their duals
and by 
\begin{itemize}
\item[]
$J:U'\to U$  the canonical  isomorphism between $U'$ and $U$;
\item[]
$\wJ:H\to H'$  the canonical  isomorphism between $H$ and $H'$.
\end{itemize}
We will always identify a Hilbert space with its bidual. Moreover we assume that the following hypotheses are satisfied:
\begin{itemize}
\item[\bf(H1)]The operator $A : \Do(A)\subset H\to H$ is the infinitesimal generator of a strongly continuous group $e^{tA}$ on $H$. 
\footnote{
Thus its adjoint $A^*: \Do(A^*)\subset H'\to H'$ generates a  group $e^{tA^*}=(e^{tA})^*$ on $H'$.
}
\item[\bf(H2)]
$B \in \Cont(U,\Do(A^*)')$.
\item[\bf(H3)]
Given $T>0$, there exists a positive constant $c_1(T)$ such that
\begin{equation*}
\int_0^T\|B^*e^{-tA^*}x\|^2_{U'}\ud x\le c_1(T)\|x\|^2_{H'}
\end{equation*}
for all $x\in \Do(A^*)$. 
\item[\bf(H4)]
There exists a number $T_0>0$ and a positive constant $c_2(T_0)$ such that
\begin{equation*}
c_2(T_0)\|x\|_{H'}^2\le \int_0^{T_0}\|B^*e^{-tA^*}x\|^2_{U'}\ud t
\end{equation*}
for all $x\in \Do(A^*)$. 
\end{itemize} 
\pagebreak
\begin{rks}${}$
\begin{itemize}
\item
$\Do(A^*)'$ denotes the dual of $\Do(A^*)$, which is a Hilbert space, provided with the  norm 
$\|x\|^2_{\Do(A^*)}:=\|x\|^2_{H'}+\|A^*x\|^2_{H'}$.
Moreover,
\begin{equation*}
\Do(A^*)\subset H' \quad \Longrightarrow \quad H \subset \Do(A^*)'
\end{equation*}
with dense and continuous embeddings.
We denote by $B^*\in \Cont(\Do(A^*),U')$ the adjoint of $B$. This implies the existence of a complex number $\lambda$ (in the resolvent set of $-A$) and a bounded operator $E\in \Cont(U,H)$ such that
\begin{equation*}
B^*=E^*(A+\lambda I)^*.
\end{equation*}
If $B\in \Cont(U,H)$, we say that $B$ is \emph{bounded}. This is the case with a distributed control. Otherwise, we say that $B$ is \emph{unbounded}, which covers the case of a boundary control.
\item
In the examples, the inequality in (H3) represents a trace regularity result (see \cite{LT1983}). It is usually called the \emph{direct inequality}.
Thanks to the assumptions (H1)-(H2), if this inequality satisfied for one $T>0$, then it is satisfied for all $T>0$ (up to a change of the positive constant). Moreover, the estimation remains true  if we integrate on $(-T,T)$. This inequality is extended to all $x\in H'$ by density and the function $t\mapsto B^*e^{-tA^*}x$ can be seen as an element of $L^2_{\text{loc}}(\R;U').$
\item
The inequality of (H4)  is usually called the \emph{inverse} or \emph{observability inequality}. Obviously it remains true if we integrate on $(0,T_1)$ with $T_1>T_0$ but need not be true if $0<T_1<T_0$.
\end{itemize}
\end{rks}

Let us recall how the feedback of \cite{K97} is constructed. We fix a number $\om>0$, set
\begin{equation*}
\Tom:=T_0+\frac{1}{2\om},
\end{equation*}
and we introduce the following weight function (see Figure 1) on the interval $[0,\Tom]$:
\begin{equation*}
e_{\omega}(s):= 
\left\{
\begin{aligned}
&e^{-2\omega s} \quad \text{if} \quad  0\le s \le T_0 \\
&2\omega e^{-2\omega T_0}(T_{\omega}-s) \quad \text{if } \quad  T_0 \le s \le T_{\omega}. 
\end{aligned}
 \right.
\end{equation*}
\begin{figure}  \label{fig} 
\begin{tikzpicture}[scale=3]
\draw[->] (-0.5,0) -- (2.5,0); 
\draw[->] (0,-0.5) -- (0,1.5);
\draw[dashed] (1,-0.5) -- (1,1.5);
\draw[dashed] (-0.5,0.135335283)--(2.5,0.135335283);
\draw [domain=0:1,smooth,thick] plot (\x,{exp(-2*\x)});
\draw[thick] (1,0.135335283)--(1.5,0);
\draw (-0.1,0) node [below] {$0$};
\draw(0.9,0) node [left,below] {$T_0$};
\draw(1.5,0) node [left,below] {$T_{\omega}$};
\draw(-0.165,0.2) node {$e^{-2\omega T_0}$};
\draw(-0.05,1) node {$1$};
\end{tikzpicture}
\caption{Weight function $\eom$}
\end{figure}
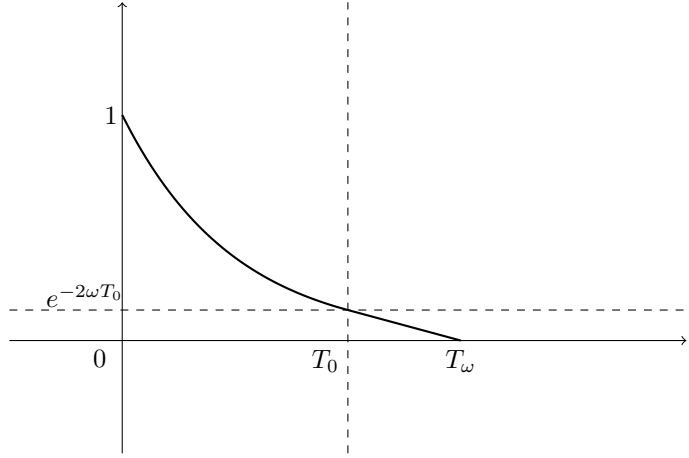
Thanks to (H3) and (H4), the relation
\begin{equation}\label{Lom}
\la\Lom x, y\ra_{H,H'}:=\int_0^{\Tom} \eom(s) \la JB^*e^{-sA^*}x, B^*e^{-sA^*}y \ra_{U,U'} \ud s
\end{equation}
defines a self-adjoint operator $\Lom \in L(H',H)$ which is bounded from below. In particular, $\Lom$ is invertible and we denote by $\ILom$ its inverse. The feedback operator is defined as
\begin{equation}\label{explicit}
F:=-B^*J\ILom.
\end{equation}

Now we recall an equation satisfied by $\Lom$. This relation will be essential in the analysis of the well-posedness of the closed-loop problem and the decay rate of its solutions.
Differentiating under the integral in \eqref{Lom}, we can prove that $\Lom$ is a solution to the following algebraic Riccati equation :
\begin{multline}\label{ARE}
\la\Lom A^*x,y\ra_{H,H'}+\la\Lom x,A^*y\ra_{H,H'}\\
+\la C\Lom x,\widetilde{J}C\Lom y \ra_{H,H'}-\la JB^*x,B^*y\ra_{U,U'}=0
\end{multline}
for all $x,y\in \Do(A^*)$. In the above equation, the operator $C\in \Cont(H)$ is the square root of the positive self-adjoint operator $L\in \Cont(H)$ defined by
\begin{equation*}
(Lx,y)_H:=-\int_0^{\Tom} \eom'(s)\la JB^*e^{-sA^*}\Lom^{-1}x, B^*e^{-sA^*}\Lom^{-1}y \ra_{U,U'} \ud s
\end{equation*}
for all $x$ and $y$ in $H$. 
\begin{rk}
The inequality $-\eom'(s)\ge 2\omega \eom(s)$ yields
\begin{equation}\label{CLom}
C^*\wJ C \ge 2\omega \Lom^{-1}.
\end{equation}
\end{rk}
\section{On the well-posedness of the closed-loop problem}\label{WP}
\subsection{Quick justification of the well-posedness}
In this paragraph, we give a justification of the well-posedness of the closed-loop problem \eqref{openloop}-\eqref{feedback} with the explicit feedback \eqref{explicit}. More precisely, we prove that if we replace $A$ by a suitable extension $\wA$ (that will be precised in the next definition), the closed-loop operator $A+BF$ generates a group.
\begin{deff}
We denote by $\wA \in \Cont(H,\Do(A^*)')$ the adjoint of $A^*$ seen as a bounded operator between $\Do(A^*)$ and $H'$. In particular,
\begin{equation*}
\la\wA x,y\ra_{\Do(A^*)',\Do(A^*)}=\la x,A^*y\ra_{H,H'}
\end{equation*}
for all $x \in H$ and $y\in \Do(A^*)$.
\end{deff}
\begin{rk}
The operator $\wA$ is an extension of $A$ i.e. the two operators coincide on $\Do(A)$. Moreover, such an extension is unique by the density of $\Do(A)$ in $H$.
\end{rk}
\begin{thm}\label{thmWP}
The operator 
\begin{equation*}
A_U:=\wA-BJB^*\ILom,\qquad \Do(A_U)=\Lom \Do(A^*)
\end{equation*}
is the infinitesimal generator of a strongly continuous group $U(t)$ in $H$.
\end{thm}
\begin{proof}
The operator 
\begin{equation*}
A_V:=-A^*-C^*\wJ C\Lom,\qquad \Do(A_V)=\Do(A^*)
\end{equation*}generates a strongly continuous group V(t) on $H'$. Indeed, this is a classical perturbation result for the generator of (semi)groups (see \cite[p. 22 and p. 76]{Pazy}) since the perturbation $-C^*\wJ C\Lom$ is a bounded operator in $H'$. By another classical result on semigroups (see \cite[p. 43 and p. 59]{EN} ), the conjugated operator
\begin{equation*}
\Lom A_V \ILom,\qquad \Do(\Lom A_V\ILom)=\Lom \Do(A^*)
\end{equation*}
is the generator of a strongly continuous group on $H$ defined by $\Lom V(t)\ILom$. 

Let $z \in \Lom \Do(A^*)$ (i.e. $\ILom z \in \Do(A^*)$) and $y \in \Do(A^*)$. From the  Riccati equation \eqref{ARE},  
\begin{equation*}
\la z,A^*y \ra_{H,H'} -\la B^*\ILom z, B^* y\ra_{H,H'}
=\la \Lom(-A^*-C^*\wJ C\Lom)\ILom z, y\ra_{H,H'}.
\end{equation*}
The definition of the extension $\wA$ and the hypothesis (H2) on $B$ yield
\begin{equation*}
\la (\wA-BJB^*\ILom)z,y \ra_{\Do(A^*)',\Do(A^*)}\\
=\la \Lom(-A^*-C^*\wJ C\Lom)\ILom z,y \ra_{\Do(A^*)',\Do(A^*)}.
\end{equation*}
Since this identity is true for all $y\in \Do(A^*)$,
\begin{equation*}
A_U z=\Lom A_V \ILom z \in \Do(A^*)' .
\end{equation*}
In fact, this identity is true in $H$ since the right member is an element of $H$.

Hence $A_U$ and $\Lom A_V \ILom$ coincides on $\Lom \Do(A^*)$, i.e. setting
\begin{equation*}
U(t):=\Lom V(t) \ILom,
\end{equation*}
the operator $A_U$ with domain $\Do(A_U)=\Lom \Do(A^*)$ is the generator of $U(t)$.
\end{proof}
\begin{rk}
Using the Riccati equation more deeply, it is possible to justify (see \cite{AV1}) the following variation of constants formula :
\begin{equation*}
U(t)x_0=e^{tA}x_0-(A+\lambda I)\int_0^t e^{(t-r)A}EJB^*\Lom^{-1}U(r)x_0\ud r,
\end{equation*}
where for each fixed $t$, the integrand is well defined as an element of $L^2_{\text{loc}}(\R,H)$. The integral is an element of $\Do(A)$ and the second term of the right member is continuous in $t$ (\cite[pp. 459-460]{BDP}). This formula is true for all $x_0 \in H$ and does not involve the extention $\wA$.
\end{rk}
\subsection{The domain of $A_U$.}
We have seen that the domain of $A_U$ is
\begin{equation*}
\Lom \Do(A^*). 
\end{equation*}
\emph{
Is it possible to link this abstract space to $\Do(A)$ or more generally to $\Do(A^k)$, where $k$ is a positive integer?
}

\bigskip
The answer relies on the nature of the control operator $B$. More precisely, if the control is bounded  then the domain of $A_U$ is exactly $\Do(A)$. This is coherent since the construction of the above paragraph is not ``necessary'' in that case, the perturbation $BF$ being a bounded operator in $H$. The situation is more complicated with an unbounded control operator : we will see through an example that in general $\Lom \Do(A^*)$ is not included in $\Do(A)$.

\bigskip \pagebreak
{\bf Bounded control operators.}
\begin{prop}\label{dombounded}
Assume that $B \in \Cont(U,H)$. Then
\begin{equation*}
\Do (A_U) =\Lom \Do(A^*)= \Do (A).
\end{equation*}
\end{prop}
\begin{proof}The inclusion $\Lom \Do(A^*) \subset \Do (A)$ is a consequence of the Riccati equation. Let $z \in \Lom \Do(A^*)$, i.e. $z=\Lom x$ for some $x\in \Do(A^*)$. Then, for all $y \in \Do(A^*)$, equation \eqref{ARE} implies that
\begin{align*}
\la z,A^*y\ra=&-\la A^*\Lom^{-1} z,\Lom y \ra-\la Cz,\wJ C\Lom y\ra+\la JB^*\Lom^{-1}z,B^* y\ra\\
=&-\la(\Lom A^*\Lom^{-1}+\Lom C^*\wJ C-BJB^*\Lom^{-1})z,y \ra.  
\end{align*}
From the definition of the adjoint operator, $z\in \Do(A)$ and 
\begin{equation*} 
Az=-(\Lom A^*\Lom^{-1}+\Lom C^*\wJ C-BJB^*\Lom^{-1})z.
\end{equation*}

\bigskip
For the inclusion $\Do(A)\subset \Lom \Do(A^*)$, we use a method of H. Zwart \cite{Zwart2007}. Set
\begin{equation*}
A_1:=A+\Lom C^*\wJ C- BJB^*\Lom^{-1},\qquad \Do(A_1)=\Do(A).
\end{equation*}
Then,
\begin{equation*}
A_1^*=A^*+C^*\wJ C\Lom-\Lom^{-1}BJB^*,\qquad  \Do(A_1^*)=\Do(A^*)
\end{equation*}
We can rewrite the Riccati equation \eqref{ARE} as
\begin{equation*}
\la A^*x,\Lom y\ra+\la \Lom x,A_1^*y\ra=0,
\end{equation*}
for all $x,y \in \Do(A^*)$.\\
We can find a complex number $s$ such that 
\begin{equation*}
s \in \rho(-A)\cap\rho(A_1^*). 
\end{equation*}
Indeed, the operators $-A$ and $A_1^*$ are the generators of strongly continuous groups. From the theorem of Hille-Yosida, their resolvent sets are  unions of two disjoint half-planes.\\
From the first inclusion, for each $x \in \Do(A^*)$, $\Lom x \in \Do(A)$ and 
\begin{align*}
A\Lom x&= -(\Lom A^*\Lom^{-1}+\Lom C^*\wJ C-BJB^*\Lom^{-1})\Lom x\\
&=-\Lom(A^*+C^*\wJ C\Lom-\Lom^{-1}BJB^*)x\\
&=-\Lom A_1^*x,
\end{align*}
which implies that
\begin{equation*}
(sI+A)\Lom x=\Lom (sI-A_1^*)x.
\end{equation*}
As $s\in \rho(A_1^*)\cap\rho(-A)$, the operators $(sI+A)$ and $(sI-A_1^*)$ are invertible. Multiplying the above relation on the left by $(sI+A)^{-1}$ and on the right by  $(sI-A_1^*)^{-1}$, we get
\begin{equation*}
\Lom (sI-A_1^*)^{-1} =(sI+A)^{-1}\Lom \qquad \text{on }H.
\end{equation*}
On $\Do(A)$, we have
\begin{align*}
\Lom^{-1}&=\Lom^{-1}(sI+A)^{-1}\Lom\Lom^{-1}(sI+A)\\
&=\Lom^{-1}\Lom(sI-A_1^*)^{-1}\Lom^{-1}(sI+A)\\
&=(sI-A_1^*)^{-1}\Lom^{-1}(sI+A).
\end{align*}
Thus, $\ILom\Do(A)\subset \Do(A_1^*)=\Do(A^*)$ i.e. $\Do(A)\subset \Lom \Do(A^*)$.
\end{proof}
{\bf Unbounded control operators.} We return to the more general case of an unbounded control operator i.e. 
\begin{equation*}
B\in \Cont(U,D(A^*)').
\end{equation*}
Set $\xi_0\in \Do(A^*)$. Let us explain how to compute $\Lom \xi_0$. We follow the method described in \cite[p. 1603]{K97} in the case of the wave equation with a Dirichlet boundary control and write it in an abstract framework. We can notice the similarity with the computation of the control in the Hilbert Uniqueness Method \cite{Lions1988}.
\begin{itemize}
\item[\textbullet]
We first solve the homogeneous problem
\begin{equation*}
\begin{cases}
\xi'(t)=-A^*\xi(t),\qquad 0\le t\le \Tom,\\
\xi(0)=\xi_0.
\end{cases}
\end{equation*}
The solution $\xi(t)$ (which is continuously differentiable because $\xi_0 \in \Do(A^*)$) is given by 
\begin{equation*}
\xi(t)=e^{-tA^*}\xi_0.
\end{equation*}
\item[\textbullet]
We consider the control 
\begin{equation*}
u(t):=\eom(t)JB^*\xi(t)=\eom (t)JE^*(A+\lambda I)^*\xi(t) \in \mathcal{C}(\R,U).
\end{equation*}
Remark that  $u(\Tom)=0$ because $\eom(\Tom)=0$.
\item[\textbullet]
Then, we solve the inhomogeneous backward problem
\begin{equation*}
\begin{cases}
y'(t)=Ay(t)+Bu(t), \qquad 0\le t\le \Tom,\\
y(\Tom)=0,
\end{cases}
\end{equation*}
whose mild solution (see \cite[pp. 459-460]{BDP}) is given by 
\begin{equation*}
y(t)=-(A+\lambda I)\int_t^{\Tom} e^{(t-s)A}Eu(s)\ud s, \qquad 0\le t\le \Tom.
\end{equation*}
This function is continuous on $[0,\Tom]$ with values in $H$.
\item[\textbullet]
Finally,
\begin{equation*}
\Lom \xi_0=-y(0).
\end{equation*}
Indeed, for $\phi_0 \in \Do(A^*)$  we have 
\begin{align*}
\la y(0),\phi_0\ra&=-\int_0^{\Tom}\la u(s), B^*e^{-sA^*}\phi_0\ra \ud s\\
&=-\int_0^{\Tom}  \eom(s)\la JB^*e^{-sA^*} \xi_0, B^*e^{-sA^*}\phi_0\ra \\
&=-\la \Lom \xi_0,\phi_0 \ra.
\end{align*}
The conclusion is a consequence of the density of $\Do(A^*)$ in $H'$.
\end{itemize}

Now, let us analyze how the regularity of $\Lom\xi_0=-y(0)$ depends on the regularity of $\xi_0$. 
Assume that $\xi_0 \in D((A^*)^2)$ in order to have $u(t)\in \mathcal{C}^1(\R,U)$. We set 
\begin{align*}
z_{\lambda}(t)&:= e^{\lambda t}\int_t^{\Tom} e^{(t-s)A}Eu(s)\ud s\\
&=\int_t^{\Tom}e^{(t-s)(A+\lambda I)}E\wu(s)\ud s,
\end{align*}
where 
\begin{equation*}
\wu(s):=e^{\lambda s}u(s).
\end{equation*}
We recall that $\lambda$ lies in the resolvent set of $-A$ and remark that $\wu(\Tom)=0$.\\ An integration by parts yields
\begin{align*}
z_{\lambda}(t)&=\int_t^{\Tom}e^{(t-s)(A+\lambda I)}E\wu(s)\ud s\\
&=\int_t^{\Tom}(A+\lambda I)e^{(t-s)(A+\lambda I)}(A+\lambda I)^{-1}E\wu(s)\ud s\\
&=(A+\lambda I)^{-1}\int_t^{\Tom}e^{(t-s)(A+\lambda I)}E\wu'(s)\ud s+ (A+\lambda I)^{-1}E\wu(t).
\end{align*}
Thus, 
\begin{align*}
y(0)&=-(A+\lambda I)z_{\lambda}(0)\\
&=\underbrace{-\int_0^{\Tom}e^{-s(A+\lambda I)}E\wu'(s)\ud s}_{\in \Do(A)} \overbrace{-E\wu(0)}^{\in \Do(A)?}.
\end{align*}
The first term on the right side of the above identity is in  $\Do(A)$ (see e.g., \cite[pp. 459-460]{BDP}). Hence, $y(0)$ belongs to $\Do(A)$ if and only if the second term belongs to $\Do(A)$. Let us take a look at this last term
\begin{equation*}
E\wu(0)=Eu(0)=EJB^*\xi_0=EJE^*(A+\lambda I)^*\xi_0.
\end{equation*}
through an example.
\begin{ex} 
We consider the wave equation on a sufficiently smooth domain $\Omega$ with a Dirichlet boundary control acting on the entire boundary $\Gamma$ :
\begin{equation*}
\begin{cases}
y''-\Delta y=0 \quad \text{in} \quad (0,\infty)\times \Omega,\\
y=u \quad  \quad \text{in} \quad (0,\infty)\times \Gamma,\\
y(0)=y_0, \quad y'(0)=y_1 \quad \text{in} \quad \Omega.
\end{cases}
\end{equation*}
This system can be rewritten in the abstract form \eqref{openloop} as follows.  We set
\begin{align*}
&H=H^{-1}(\Om)\times L^2(\Om);\\
&H'=H^{1}_0(\Om)\times L^2(\Om);\\
&U=U'=L^2(\Gamma).
\end{align*}
Next, $-A^*$ is the wave operator on $H'$ and $A:=(A^*)^*$ with 
\begin{align*}
&\Do(A^*)=(H^2(\Omega)\cap H^1_0(\Omega)) \times H^1_0(\Omega);\\
&\Do(A)=L^2(\Omega)\times H^1_0(\Omega)
\end{align*}
(see e.g., \cite{CH1998} or \cite{K97} for details).
The operator $B^*$ is defined for $(\eta_0,\eta_1)\in \Do(A^*)$  by
\begin{equation*}
B^*(\eta_0,\eta_1):=\partial_{\nu}\eta_0\in L^2(\Gamma).
\end{equation*}
Moreover, for $u \in L^2(\Gamma)$, we set
\begin{equation*}
Eu:=(0,-Du),
\end{equation*}
where $D:L^2(\Gamma) \to H^{1/2}(\Om)$ is the Dirichlet map, defined by
\begin{equation*}
\begin{cases}
-\Delta Du=0 \quad \text{in } \Om,\\
Du=u \quad \text{in } \Gamma.
\end{cases}
\end{equation*}
Assume that $(\xi_0,\xi_1)\in \Do(A^*)$. Then
\begin{equation*}
EJB^*(\xi_0,\xi_1)=(0,-D\partial_{\nu}\xi_0)\in H.
\end{equation*}
We seek a condition on  $\xi_0$ in order to have $(0,D\partial_{\nu}\xi_0) \in \Do(A)$ i.e. 
\begin{equation*}
D\partial_{\nu} \xi_0\in H^1_0(\Omega).
\end{equation*}
If $\xi_0 \in H^2(\Om)\cap H^1_0(\Om)$, then $\partial_{\nu} \xi_0 \in H^{1/2}(\Gamma)$ and $D\partial_{\nu} \xi_0 \in H^1(\Omega)$. By definition of $D$, the trace of $D\partial_{\nu} \xi_0$ on the boundary is 
\begin{equation*}
D\partial_{\nu} \xi_0|_{\Gamma}=\partial_{\nu} \xi_0.
\end{equation*}
That is why, in order to have $D\partial_{\nu} \xi_0\in H^1_0(\Omega)$ it is necessary and sufficient that $\partial_{\nu}\xi_0=0$.
But (see e.g., \cite[p. 217]{Brezis})
\begin{equation*}
D((A^*)^2)=\Big\{(\xi,\psi)\in H^3(\Om)\times H^2(\Om) : \xi=\Delta \xi=\psi=0 \quad \text{on }\Gamma \Big\}.
\end{equation*}
Hence, if $(\xi_0,\xi_1)\in \Do((A^*)^2)$, the normal derivative of $\xi_0$ on $\Gamma$ does not necessarily vanish.
Finally, given $(\xi_0,\xi_1)\in \Do((A^*)^2)\subset \Do(A^*)$,
\begin{equation*}
\Lom(\xi_0,\xi_1)\in \Do(A) \iff \partial_{\nu}\xi_0=0.
\end{equation*}
In this example of boundary control,
\begin{equation*}
\Lom\Do(A^*)  \not \subset \Do(A).
\end{equation*}
\end{ex}
\begin{rk}
In the case of a vibrating string, if $\Om=(0,\pi)$,  the eigenfunctions of the Laplacian with Dirichlet boundary conditions are the functions $\sin(nx)$, where $n=1,2,\ldots$ They belong to the class $\mathcal{C}^{\infty}$. Even if $\xi_0$ and $\xi_1$ are linear combinations of these functions, $\Lom(\xi_0,\xi_1)$ does not necessarily belong to  $\Do(A)$ because the normal derivative of $\xi_0$ does not necessarily vanish on $\Gamma$.
\end{rk}
\section{On the decay rate of the solutions}\label{Decay}
Let us give a proof of the exponential decay of the solutions of \eqref{openloop}-\eqref{feedback} with the explicit feedback \eqref{explicit}. The proof is different from the one given in \cite[pp. 1598-1599]{K97}: we do not use an integral representation formula fo $\ILom$. It is closer to the finite dimensional case (see \cite[p. 1597]{K97}).
\begin{prop}\label{rateomega}
There exists a positive constant $c$ such that for each initial data $x_0$,
\begin{equation*}
\|U(t)x_0\|_H\le ce^{-\omega t}\|x_0\|_H,
\end{equation*}
for all $t\ge 0$.
\end{prop}
\begin{proof}
At first, let $x_0\in \Do(A_U)=\Lom  \Do(A^*)$ and set $x(t):=U(t)x_0$. We have already seen that $A_U$ coincides with $\wA-BJB^*\ILom$ on $ \Do(A_U)$.
With this regularity for the initial data, $x(t)$ is differentiable on $\R$ and
\begin{align*}
\frac{1}{2}\frac{\ud}{\ud t}\la \ILom x(t),x(t) \ra_{H',H}&=\la \ILom x(t),x'(t)\ra_{H',H}\\
&=\la \ILom x(t),A_Ux(t)\ra_{H',H}\\
&=\la\underbrace{\ILom x(t)}_{\in  \Do(A^*)},\underbrace{(\wA-BJB^*\ILom)x(t)}_{\in H}\ra_{H',H}\\
&=\la\ILom x(t),\wA x(t)\ra_{ \Do(A^*), \Do(A^*)'}\\
&\qquad \qquad -\la\ILom x(t),BJB^*\ILom x(t)\ra_{ \Do(A^*), \Do(A^*)'}\\
&=\la A^*\ILom x(t),x(t)\ra_{H',H}-\la JB^*\ILom x(t),B^*\ILom x(t)\ra_{U,U'}.
\end{align*}

For all $x,y\in D(A_U)$, we have
\begin{equation*}
\la A^*\Lom^{-1}x,y\ra_{H',H}+\la x,A^*\Lom^{-1}y\ra_{H,H'}+\la \wJ Cx,Cy\ra_{H',H}-\la JB^*\Lom^{-1}x,B^*\Lom^{-1}y\ra_{U,U'}=0.
\end{equation*}
Thus
\begin{align*}
\frac{\ud}{\ud t}\la \Lom^{-1}x(t),x(t) \ra_{H',H}&=-\la \wJ Cx(t),Cx(t)\ra_{H',H}-\la JB^*\Lom^{-1}x(t),B^*\Lom^{-1}x(t)\ra_{U,U'}\\
&\le -\la \wJ Cx(t),Cx(t)\ra_{H',H}\\
&\le -2\omega \la \Lom^{-1}x(t),x(t) \ra_{H',H}
\end{align*}
where the last inequality is a consequence of \eqref{CLom}.

Finally the above estimations yield 
\begin{equation*}
\la\Lom^{-1}x(t),x(t)\ra_{H',H}\le e^{-2\omega t}\la\Lom^{-1}x_0,x_0\ra_{H',H}, \qquad t\ge 0. 
\end{equation*}
This estimation remains true for $x_0\in H$ by density of $D(A_U)$ in $H$.
We conclude by noticing that, thanks to (H3) and (H4), the quantity $(\la \ILom x,x\ra_{H',H})^{1/2}$ defines a norm on $H$ which is equivalent to $\|.\|_H$.
\end{proof}

Let us explain on a finite-dimensional example why we can expect  a better decay rate than $\omega$ in some cases.
\begin{ex}
We consider the the system governed by
\begin{equation*}
\begin{cases}
y''(t)+y(t)=u(t),\qquad t\ge 0\\
y(0)=y_0,\quad y'(0)=y_1.
\end{cases}
\end{equation*}
Setting $x(t)=\begin{pmatrix}y(t)\\y'(t)\end{pmatrix}$,
$x_0=\begin{pmatrix}y_0\\y_1\end{pmatrix}$,
$A=\begin{pmatrix}0&1\\-1&0\end{pmatrix}$ and
$B=\begin{pmatrix}0\\1\end{pmatrix}$,
this system is equivalent to \eqref{openloop}. Hypotheses (H1) to (H3) are obviously satisfied while hypothesis (H4) follows from the observability of the pair $(-A^*,B^*)$ that one can check via the rank condition.

Now we compute the feedback operator \eqref{explicit} replacing $\Lom$ by the slightly different operator (in order to make the computations easier) 
\begin{align*}
\widetilde{\Lom}:=&\int_0^{T_0}e^{-2\om t}e^{-tA^*}BB^*e^{-tA}\ud t\\
=&
\begin{pmatrix}
\int_0^{T_0} e^{-2\om t}\sin^2t\ud t&-\int_0^{T_0} e^{-2\om t}\sin t\cos t\ud t\\
-\int_0^{T_0} e^{-2\om t}\sin t\cos t\ud t&\int_0^{T_0} e^{-2\om t}\cos^2t\ud t
\end{pmatrix}.
\end{align*}
For some particular values of $T_0$, the coefficients of the above matrix are particularly simple and so are the coefficients of the closed loop operator. Chossing $T_0=k\pi$, where $k$ is a positive integer, we obtain
\begin{equation*}
A_U=A-BB^*\widetilde{\Lom}^{-1}=
\begin{pmatrix}
0&1\\
-1-\frac{4\om^2}{1-e^{-2\om k\pi}} & \frac{-4\om}{1-e^{-2\om k\pi}}
\end{pmatrix}.
\end{equation*}

In finite dimension, 
\begin{equation*}
\text{growth bound of $A_U$}\, = \, \max\{\text{Re}(\lambda),\,\lambda \in \text{Sp}(A_U)\}.
\end{equation*}
The eigenvalues of $A_U$ are complex and conjugated since this matrix is real and the discriminant of the characteristic polynomial is negative :
\begin{equation*}
(\text{tr}A_U)^2-4\text{det}A_U=\frac{8e^{-2\om k\pi}}{(1-e^{-2\om k\pi})^2}\Big((1+\frac{(2\om)^2}{2}-\text{ch}(2\om k\pi)\Big)<0
\end{equation*}
Indeed, $\text{ch}(k\pi x)>\text{ch}(x)>1+\frac{x^2}{2}$ for all $x>0$. Hence, denoting by $\lambda$ and $\bar{\lambda}$ the eigenvalues, we obtain
\begin{equation*}
\text{Re}\lambda=\text{Re}\bar{\lambda}=\frac{1}{2}\text{tr}A_U=\frac{-2\om}{1-e^{-2\om k \pi}}<-2\om<-\om.
\end{equation*}
Therefore, some choices of $T_0$ yield a decay rate that is at least twice better than the one obtain in Theorem  \eqref{rateomega}.
\end{ex}
We return to the general case. The aim of the following paragraph is to show that by replacing $T_0$ by $T\ge T_0$ in the definition of $\Lom$ \eqref{Lom}, it is possible  to have a larger decay rate of the solutions. Moreover, for dissipative systems, this decay rate approaches ``quickly'' the value $-2\omega$ as $T$ increases. This may explain the larger decay rate observed in some numerical and physical experiments. 

Let $c\ge 1$ and $\gamma\in \R$ be two constants such that 
\begin{equation*}
\forall t\ge 0, \qquad \|e^{-tA^*}\|\le ce^{\gamma t}.
\end{equation*}
In the sequel, the value $\omega>0$ is fixed and we denote by $\LomT$ the operator obtained in \eqref{Lom} by replacing $T_0$ by $T\ge T_0$ (we will also write $\eomT$ for  the corresponding weight function). Thanks to hypotheses (H3) and (H4), this operator has the same properties as $\Lom$. In particular it is invertible. Note that we also have to replace $C$ by an operator $C_T$ (see the definition of $C$ in section \ref{Hyp}). We can repeat the method of section \ref{WP} to prove the well-posedness of the closed-loop problem with the feedback $F=-JB^*\ILomT$ in the framework of semigroups.

\begin{thm}\label{betterrate}
For a fixed  $T\ge T_0$, the semigroup $U_T(t)$ generated by $\wA-BJB^*\ILomT$ satisfies 
\begin{equation*}
\|U_T(t)\|\le c' \exp\big((-2\omega+\gamma+\alpha \varphi(T))t\big), \qquad t\ge 0,
\end{equation*}
where 
\begin{equation*}
\varphi(T):=\exp\big(\gamma T-2\omega(T-T_0)\big), \qquad 
c':=c\|\LomT\|\|\ILomT\|
\end{equation*}
and $\alpha$ is a positive constant that depends only on $T_0$ and $\omega$.
\end{thm}
\begin{rk}
This estimation of the decay rate of the solutions of the closed-loop problem may be worst than the one given by Proposition \ref{rateomega} (at least $\omega$). But in the case of conservative or dissipative systems, we have $\gamma=0$. Consequently, if $T-T_0$ is sufficiently large, 
\begin{equation*}
\gamma -2\omega +\alpha\varphi(T) \approx -2\omega,
\end{equation*}
i.e. the decay rate of the solutions is approximately $-2\omega$.
\end{rk}
\begin{proof}
Instead of working with the semigroup $U_T(t)$, we work with one of its conjugates (whose generator is easier to manipulate, see the proof of Theorem \ref{thmWP})
\begin{equation*}
V_T(t):=\ILomT U_T(t) \LomT.
\end{equation*}
Its generator is
\begin{equation*}
A_{V_T}=-A^*-C^*_T\wJ C_T\LomT, \qquad \Do(A_{V_T})=\Do(A^*).
\end{equation*}
We recall from the definition of the operator $C_T$ that 
\begin{equation*}
C^*_T\wJ C_T\LomT=\ILomT \LomT' \in \Cont(H'),
\end{equation*}
where $\LomT'\in \Cont(H',H)$ is the self-adjoint, positive definite operator defined by
\begin{align*}
\la \LomT' x,y\ra:=&-\int_0^{T+1/{2\omega}} \eomT'(s)\la JB^*e^{-sA^*}x,B^*e^{-sA^*}y\ra \ud s\\
=&2\om\int_0^{T}e^{-2\om s}\la JB^*e^{-sA^*}x,B^*e^{-sA^*}y\ra \ud s\\
&\quad  +2\om\int_T^{T+1/{2\om}}\eomT(s) \la JB^*e^{-sA^*}x,B^*e^{-sA^*}y\ra \ud s\\
&\quad -2\om\int_T^{T+1/{2\om}}\eomT(s) \la JB^*e^{-sA^*}x,B^*e^{-sA^*}y\ra \ud s\\
&\quad +2\om e^{-2\om T}\int_T^{T+1/{2\om}}\la JB^*e^{-sA^*}x,B^*e^{-sA^*}y\ra \ud s\\
=:&\,2\om \la \LomT x,y\ra+\la\RomT x,y \ra,
\end{align*}
with $\RomT \in \Cont(H',H)$ self-adjoint and positive. 
Hence
\begin{equation*}
A_{V_T}=-A^*-2\om I-\ILomT\RomT.
\end{equation*}

The operator $-A^*-2\om I$ with domain $\Do(A^*)$ is the generator of a semigroup and we have the following estimation : 
\begin{equation*}
\|e^{t(-A^*-2\om I)}\|=e^{-2\om t}\|e^{-tA^*}\|\le ce^{(\gamma-2\om)t}, \qquad t\ge 0.
\end{equation*}
In order to have an estimation for the semigroup $V_T(t)$, we are going to apply a classical (bounded) perturbation result
\footnote{
If $P$ is a bounded operator in $H'$, the operator $-A^*+P$ with domain $\Do(A^*)$ is the generator of a semigroup in $H'$ and  $\|e^{(-A^*+P)t}\|\le ce^{(\gamma +c\|P\|)t},\, t\ge 0$ (see \cite[p. 76]{Pazy}).
}.
The idea is that the growth of the semigroup generated by a perturbated operator can be expressed in term of the norm of the perturbation. Let us estimate the norm of the bounded perturbation $\ILomT\RomT$.

For all $x\in H'$,
\begin{equation*}
c_2(T_0)e^{-2\om T_0}\|x\|^2_{H'}\le \la \LomTO x,x\ra_{H,H'} \le \la \LomT x,x\ra_{H,H'},
\end{equation*}
where $c_2(T_0)$ is the positive number given by (H4).
Hence, 
\begin{equation*}
\|\ILomT\|\le \frac{e^{2\om T_0}}{c_2(T_0)}.
\end{equation*}

We remark that the weight function in the operator $\RomT$ satisfies 
\begin{equation*}
0\le 2\om(e^{-2\om T}-\eomT(s))=2\om e^{-2\om T}(1-2\om (T-s))\le 2\om e^{-2\om T}
\end{equation*}
for $T\le s \le T+1/{2\om}.$ Thus, 
\begin{equation*}
\|\RomT\|\le  c_1(1/{2\om})ce^{\gamma T}2\om e^{-2\om T},
\end{equation*}
where $c_1(1/{2\om})$ is a positive constant that depends only on $\om$ (see (H3) and the second remark below the hypotheses)
\footnote{
The term``$ce^{\gamma T}$'' appears when we transpose the direct inequality from the interval $(0,1/{2\omega})$ to $(T,T+1/{2\omega})$.
}.

Applying the perturbation result, we obtain
\begin{equation*}
\|V_T(t)\|\le c\exp\Big(\big(-2\om + \gamma +2\om c^2c_1(1/{2\om}) c_2(T_0)^{-1}e^{\gamma T+ 2\om(T_0-T)}\big)t \Big), \quad t\ge 0.
\end{equation*}
The estimation on $U_T(t)$ is a direct consequence of its relationship with $V_T(t)$.
\end{proof}
\begin{rks}${}$
\begin{itemize}
\item
The relationship of the norm of the solution of a Riccati equation and the decay rate of the solutions has  been studied in the framework of optimal control theory in \cite{BL1996}. 
\item
Integrating from $0$ to $\infty$ in the definition of $\Lom$, it is possible to  construct a feedback that leads to a decay rate of  $-2\omega$ in the dissipative case (see \cite{Urquiza2005} and also \cite{Russell1979} for an exposition of this method in the finite-dimensional case).
\end{itemize}
\end{rks}


\begin{thebibliography}{10}

\bibitem{AK1999}
{\sc F.~Alabau and V.~Komornik}, {\em Boundary observability, controllability,
  and stabilization of linear elastodynamic systems}, SIAM J. Control Optim.,
  37 (1999), pp.~521--542 (electronic).

\bibitem{BL1996}
{\sc A.~Benabdallah and M.~Lenczner}, {\em Estimation du taux de
  d{\'e}croissance pour la solution de probl{\`e}mes de stabilisation,
  application {\`a} la stabilisation de l'{\'e}quation des ondes}, RAIRO
  Mod{\'e}l. Math. Anal. Num{\'e}r., 30 (1996), pp.~607--635.

\bibitem{BDP}
{\sc A.~Bensoussan, G.~Da~Prato, M.~C. Delfour, and S.~K. Mitter}, {\em
  Representation And Control of Infinite Dimensional Systems}, Systems \&
  Control: Foundations \& Applications, Birkh{\"a}user Boston Inc., Boston, MA,
  second~ed., 2007.

\bibitem{BJCR2004}
{\sc F.~Bourquin, M.~Joly, M.~Collet, and L.~Ratier}, {\em An efficient
  feedback control algorithm for beams: experimental investigations}, Journal
  of Sound and Vibration, 278 (2004), pp.~181--206.

\bibitem{Brezis}
{\sc H.~Brezis}, {\em Analyse Fonctionnelle}, Dunod, 2005.

\bibitem{Briffaut1999}
{\sc J.-S. Briffaut}, {\em M\'ethodes num\'eriques pour le contr\^ole et la
  stabilisation rapide des grandes strucutures flexibles}, PhD thesis, \'Ecole
  Nationale des Ponts et Chauss\'ees, 1999.

\bibitem{CH1998}
{\sc T.~Cazenave and A.~Haraux}, {\em An Introduction to Semilinear Evolution
  Equations}, vol.~13 of Oxford Lecture Series in Mathematics and its
  Applications, The Clarendon Press Oxford University Press, New York, 1998.

\bibitem{Chen1979}
{\sc G.~Chen}, {\em Energy decay estimates and exact boundary value
  controllability for the wave equation in a bounded domain}, J. Math. Pures
  Appl. (9), 58 (1979), pp.~249--273.

\bibitem{EN}
{\sc K.~Engel and R.~Nagel}, {\em One-Parameter Semigroups for Linear Evolution
  Equations}, Springer-Verlag, 1999.

\bibitem{Kleinman1970}
{\sc D.~L. Kleinman}, {\em An easy way to stabilize a linear constant system},
  IEEE Transactions on Automatic Control,  (1970), p.~692.

\bibitem{K97}
{\sc V.~Komornik}, {\em Rapid boundary stabilization of linear distributed
  systems}, SIAM Journal on Control and Optimization, 35 (1997),
  pp.~1591--1613.

\bibitem{K98}
{\sc V.~Komornik}, {\em Rapid boundary stabilization of {M}axwell's equations},
  in \'{E}quations aux d{\'e}riv{\'e}es partielles et applications,
  Gauthier-Villars, {\'E}d. Sci. M{\'e}d. Elsevier, Paris, 1998, pp.~611--622.

\bibitem{KL05}
{\sc V.~Komornik and P.~Loreti}, {\em Fourier Series in Control Theory},
  Springer Monographs in Mathematics, Springer-Verlag, New York, 2005.

\bibitem{Lagnese1983}
{\sc J.~Lagnese}, {\em Decay of solutions of wave equations in a bounded region
  with boundary dissipation}, J. Differential Equations, 50 (1983),
  pp.~163--182.

\bibitem{LT1983}
{\sc I.~Lasiecka and R.~Triggiani}, {\em Regularity of hyperbolic equations
  under {$L_{2}(0,\,T;L_{2}(\Gamma ))$}-{D}irichlet boundary terms}, Appl.
  Math. Optim., 10 (1983), pp.~275--286.

\bibitem{LT1987}
\leavevmode\vrule height 2pt depth -1.6pt width 23pt, {\em Uniform exponential
  energy decay of wave equations in a bounded region with {$L_2(0,\infty;
  L_2(\Gamma))$}-feedback control in the {D}irichlet boundary conditions}, J.
  Differential Equations, 66 (1987), pp.~340--390.

\bibitem{Lions1988}
{\sc J.-L. Lions}, {\em Exact controllability, stabilizability and
  perturbations for distributed systems}, SIAM Rev.,  (1988), pp.~1--68.

\bibitem{Lukes1968}
{\sc D.~L. Lukes}, {\em Stabilizability and optimal control}, Funkcial. Ekvac.,
  11 (1968), pp.~39--50.

\bibitem{Pazy}
{\sc A.~Pazy}, {\em Semigroups of Linear Operators and Applications to Partial
  Differential Equations}, Springer-Verlag, 1992.

\bibitem{QR1977}
{\sc J.~P. Quinn and D.~L. Russell}, {\em Asymptotic stability and energy decay
  rates for solutions of hyperbolic equations with boundary damping}, Proc.
  Roy. Soc. Edinburgh Sect. A, 77 (1977), pp.~97--127.

\bibitem{Russell1979}
{\sc D.~L. Russell}, {\em Mathematics of Finite-Dimensional Control Systems},
  vol.~43 of Lecture Notes in Pure and Applied Mathematics, Marcel Dekker Inc.,
  New York, 1979.
\newblock Theory and design.

\bibitem{Slemrod}
{\sc M.~Slemrod}, {\em A note on complete controllability and stabilizability
  for linear control systems in {H}ilbert space}, SIAM J. Control,  (1974),
  pp.~500--508.

\bibitem{SGK2009}
{\sc A.~Smyshlyaev, B.-Z. Guo, and M.~Krstic}, {\em Arbitrary decay rate for
  {E}uler-{B}ernoulli beam by backstepping boundary feedback}, IEEE Trans.
  Automat. Control, 54 (2009), pp.~1134--1140.

\bibitem{Urquiza2005}
{\sc J.~M. Urquiza}, {\em Rapid exponential feedback stabilization with
  unbounded control operators}, SIAM J. Control Optim., 43 (2005),
  pp.~2233--2244.

\bibitem{AV1}
{\sc A.~Vest}, {\em Rapid stabilization in a semigroup framework}.
\newblock Preprint, arxiv:1301.5744, 2013.

\bibitem{Zwart2007}
{\sc H.~Zwart}, {\em Invertible solutions of the {L}yapunov and algebraic
  {R}iccati equation}.
\newblock Preprint, wwwhome.math.utwente.nl/{$\sim$}zwarthj/, 2007.

\end{thebibliography}
\end{document}